\documentclass[11pt,oneside]{amsart}
      \usepackage{amssymb,enumerate,bbm}

      \theoremstyle{plain}
      \newtheorem{theorem}{Theorem}[section]

      \theoremstyle{definition}

      \theoremstyle{remark}
      \newtheorem{remark}[theorem]{Remark}

      \newtheorem{example}[theorem]{Example}
			
      \theoremstyle{proposition}

      \theoremstyle{ass}
      \newtheorem{ass}[theorem]{Assumption}

      \makeatletter
      \def\@setcopyright{}

      \def\serieslogo@{}
      \makeatother

\usepackage{graphicx}
\usepackage{caption}	
\usepackage{tikz}
\usetikzlibrary{calc}
\usetikzlibrary{through}
\usetikzlibrary{decorations.markings}

\usepackage{geometry}
\usepackage{marginnote}
\usepackage{verbatim}

\usepackage{color}

\newcommand{\R}{\mathbb{R}}

\renewcommand{\d}{\delta}

\renewcommand{\a}{\alpha}

\theoremstyle{plain}

\begin{document}

\author{Peter Eichelsbacher}
\address{Fakult\"at f\"ur Mathematik, Ruhr-Universit\"at Bochum, NA 3/66, 44780 Bochum, Germany}
\email{peter.eichelsbacher@rub.de}

\author{ Matthias L\"owe}
\address{Westf\"alische Wilhelms-Universit\"at M\"unster, Fachbereich Mathematik, Einsteinstra\ss e 62, 48149 M\"unster, Germany}
\email{maloewe@math.uni-muenster.de}

\title[Moderate deviations via Lindeberg's method]{\large Lindeberg's method for moderate deviations and random summation}

\begin{abstract}
We apply Lindeberg's method, invented to prove a central limit theorem, to analyze the moderate deviations around such a central limit theorem. In particular, we will show moderate deviation principles for martingales as well as for random sums, in the latter situation both, in the case when the limit distribution is Gaussian or non-Gaussian. Moreover in the Gaussian case we show moderate deviations for random sums using bounds on cumulants, alternatively.
Finally, we also prove a large deviation principle for certain random sums.
\end{abstract}



\thanks{This research was done partly at the Mathematisches Forschungsinstitut Oberwolfach during a stay within the Research in Pairs Programme.}


\maketitle

\newcommand{\Ncal}{{\mathcal {N}}}
\newcommand{\Mcal}{{\mathcal {M}}}
\newcommand{\Ocal}{{\mathcal {O}}}
\newcommand{\Dcal}{{\mathcal {D}}}
\newcommand{\Tcal}{{\mathcal {T}}}
\newcommand{\Hcal}{{\mathcal {H}}}


\def\rr{\mathbb R}
\def\mm{\mathbb M}
\def\nn{\mathbb N}
\def\cc{\mathbb C}
\def\torus{\mathbb T}
\def\zz{\mathbb Z}
\def\qq{\mathbb Q}
\def\pp{\mathbb P}
\def\kk{\mathbb K}
\def\erw{\mathbb E}

\def\phi{\varphi }

\def\calf{{\mathcal F}}
\def\cala{{\mathcal A}}
\def\calb{{\mathcal B}}
\def\calc{{\mathcal C}}
\def\cald{{\mathcal D}}
\def\cale{{\mathcal E}}
\def\calg{{\mathcal G}}
\def\calh{{\mathcal H}}
\def\cali{{\mathcal I}}
\def\calj{{\mathcal J}}
\def\calk{{\mathcal K}}
\def\call{{\mathcal L}}
\def\calm{{\mathcal M}}
\def\caln{{\mathcal N}}
\def\calo{{\mathcal O}}
\def\calp{{\mathcal P}}
\def\calq{{\mathcal Q}}
\def\calr{{\mathcal R}}
\def\cals{{\mathcal S}}
\def\calt{{\mathcal T}}
\def\calu{{\mathcal U}}
\def\calv{{\mathcal V}}
\def\calw{{\mathcal W}}
\def\calx{{\mathcal X}}
\def\caly{{\mathcal Y}}
\def\calz{{\mathcal Z}}
\def\cala{{\mathcal A}}

\def\frA{{\mathfrak A}}
\def\frB{{\mathfrak B}}
\def\frC{{\mathfrak C}}
\def\frD{{\mathfrak D}}
\def\frE{{\mathfrak E}}
\def\frF{{\mathfrak F}}
\def\frG{{\mathfrak G}}
\def\frH{{\mathfrak H}}
\def\frI{{\mathfrak I}}
\def\frJ{{\mathfrak J}}
\def\frK{{\mathfrak K}}
\def\frL{{\mathfrak L}}
\def\frM{{\mathfrak M}}
\def\frN{{\mathfrak N}}
\def\frO{{\mathfrak O}}
\def\frP{{\mathfrak P}}
\def\frQ{{\mathfrak Q}}
\def\frR{{\mathfrak R}}
\def\frS{{\mathfrak S}}
\def\frT{{\mathfrak T}}
\def\frU{{\mathfrak U}}
\def\frV{{\mathfrak V}}
\def\frW{{\mathfrak W}}
\def\frX{{\mathfrak X}}
\def\frY{{\mathfrak Y}}
\def\frZ{{\mathfrak Z}}

\def\fra{{\mathfrak a}}
\def\frb{{\mathfrak b}}
\def\frc{{\mathfrak c}}
\def\frd{{\mathfrak d}}
\def\fre{{\mathfrak e}}
\def\frf{{\mathfrak f}}
\def\frg{{\mathfrak g}}
\def\frh{{\mathfrak h}}
\def\fri{{\mathfrak i}}
\def\frj{{\mathfrak j}}
\def\frk{{\mathfrak k}}
\def\frl{{\mathfrak l}}
\def\frm{{\mathfrak m}}
\def\frn{{\mathfrak n}}
\def\fro{{\mathfrak o}}
\def\frp{{\mathfrak p}}
\def\frq{{\mathfrak q}}
\def\frr{{\mathfrak r}}
\def\frs{{\mathfrak s}}
\def\frt{{\mathfrak t}}
\def\fru{{\mathfrak u}}
\def\frv{{\mathfrak v}}
\def\frw{{\mathfrak w}}
\def\frx{{\mathfrak x}}
\def\fry{{\mathfrak y}}
\def\frz{{\mathfrak z}}

\def\cD{{\mathcal D}}
\def\cI{{\mathcal I}}
\def\cL{{\mathcal L}}
\def\cM{{\mathcal M}}
\def\cO{{\mathcal O}}
\def\cC{{\mathcal C}}
\def\cK{{\mathcal K}}
\def\cB{{\mathcal B}}
\def\cN{{\mathcal N}}
\def\cT{{\mathcal T}}
\def\cX{{\mathcal X}}
\def\cZ{{\mathcal Z}}

\renewcommand{\a}{\alpha}
\renewcommand{\b}{\beta}

\def\tra{^{\prime}}
\def\inv{^{-1}}

\def\lra{\longrightarrow}
\def\lmt{\longmapsto}

\newcommand{\eps}{\varepsilon}
\newcommand{\id}{{\operatorname {id}}}
\newcommand{\diam}{{\operatorname {diam}}}
\newcommand{\const}{{\operatorname {const.}\,}}
\renewcommand{\b}     {\beta}
\renewcommand{\a}     {\alpha}
\newcommand{\vep}{\varepsilon}
\newcommand{\A}     {\mathbb{A}}
\newcommand{\Z}     {\mathbb{Z}}
\newcommand{\N}     {\mathbb{N}}
\renewcommand{\P}   {\mathbb{P}}
\newcommand{\D}     {\mathbb{D}}
\newcommand{\E}     {\mathbb{E}}
\newcommand{\M}     {\mathbb{M}}
\newcommand{\X}     {\mathbb{X}}
\newcommand{\V}     {\mathbb{V}}
\renewcommand{\d}   {\operatorname{d}\!}
\newcommand{\tx}    {{\widetilde x}}
\newcommand{\ty}    {{\widetilde y}}
\newcommand{\dist}  {\operatorname{dist}}
\newcommand{\sign}  {\operatorname{sign}}
\newcommand{\Schur} {\operatorname{Schur}}
\newcommand{\Sym}   {\mathfrak{S}}
\newcommand{\heap}[2]{\genfrac{}{}{0pt}{}{#1}{#2}}
\def\1{{\mathchoice {1\mskip-4mu\mathrm l}
                    {1\mskip-4mu\mathrm l}
                    {1\mskip-4.5mu\mathrm l} {1\mskip-5mu\mathrm l}}}
\newcommand{\ssup}[1] {{\scriptscriptstyle{({#1}})}}
\newcommand{\smallsup}[1] {{\scriptscriptstyle{({#1}})}}

\setcounter{section}{0}

\section{Introduction}
In 1922 Lindeberg published his article ''Eine neue Herleitung des Exponentialgesetzes in der
Wahrscheinlichkeitsrechnung'' \cite{Lindeberg:1922b} where he developed a new method to prove the Central Limit Theorem (CLT). Under the name ''replacement trick'' this technique has nowadays become a standard tool in probability theory. The key idea is to derive bounds for a suitable distance, that metricizes weak convergence, between the distribution of a standardized sum of independent and identically distributed (i.i.d., for short) random variables with existing second moments and the distribution of a standardized sum of i.i.d. Gaussian random variables with the same variances, by replacing the original variables by their Gaussian counterparts and controlling the differences. Over the years Lindeberg's method has found a variety of applications, see e.g. \cite{Bolthausen:1982} for a version that proves Berry-Esseen type results for martingales, \cite{Bardetetal:2008} for central limit theorems for dependent processes and \cite{Toda:2012} for an application of Lindeberg's method for proving convergence in distribution of random sums to a Laplace distribution. In particular, recently a renaissance of the replacement trick could be observed, when it was applied to obtain universality results in the context of random matrix theory, especially for local eigenvalue statistics (see \cite{Chatterjee:2006} and \cite{Tao/Vu:2009}). Here a fascinating result is that the limiting distributions of many local eigenvalues statistics of Hermitian random matrices with independent entries otherwise are universal as long as their first four moments agree with that of a standard Gaussian distribution.

The starting point of the current paper is the observation that a standard CLT and a moderate deviation principle (MDP) for i.i.d. random variables share the form of the asymptotic distribution, at least on a logarithmic scale. It is therefore natural to conjecture, that a technique that is successful when proving a CLT is also promising in the context of moderate deviations. To warm up, we will show that this ansatz is in principle justified for sums of i.i.d. random variables with sufficiently many moments in Section 2. However, the moderate deviations for sums of i.i.d. random variables has of course already been thoroughly analyzed in many contributions, see e.g. the standard textbook \cite{Dembo/Zeitouni:LargeDeviations}, or \cite{Eichelsbacher/Loewe:2001}. An interesting observation will be that the proof in the i.i.d. case
can immediately be generalized to martingales. However, in Section 3 our main focus will be on deriving MDPs for random sums.

Random sums are random variables of the form
\begin{equation} \label{randomsum}
S_\nu= \sum_{i=1}^\nu X_i
\end{equation}
where the $(X_i)_i$ are in most of the cases i.i.d. random variables and $\nu$ is another random variable that takes values in the natural numbers $\N$ or $\N_0$ and is independent of the $(X_i)_i$. More precisely $\nu$ is to be thought as $\nu_p$, where $p$ is an extra parameter and the expectation of $\nu_p$ diverges to infinity, when $p$ becomes large. Since random sums are both, theoretically interesting and important in applications (e.g. in insurance mathematics) there is a vast literature on random sums. Standard references are, e.g. \cite{Gnedenko/Korolev:1996} or \cite{Kalashnikov:1997} and many of the references cited therein. As nice recent examples we mention
\cite{Fleischmann/Wachtel:2008}, where the summation variable is given by the number of children of a supercritical Galton-Watson process and \cite{Doebler:2015}, where Stein's method is applied to obtain CLTs for random sums. In a nutshell the situation for the CLT is the following: First of all the kind of CLT we can expect depends on the question, whether or not the $(X_i)_i$ are purely positive (which makes sense in the context of insurance mathematics, where they model the losses an insurance company faces) or not. But even, if the $(X_i)_i$ are centered and have finite, non vanishing second moment, the limiting distribution of the appropriately scaled process of the $S_\nu$ may or may not be Gaussian. The decisive property for this question is, how strongly $\nu$ is asymptotically concentrated in its expected value. In case the limiting distribution in not normal, there is a whole zoo of possible limiting distribution, cf. \cite{Gnedenko/Korolev:1996}, Chapters 3 and 4.

The main part of this article, this Sections 3 and 4, will be devoted to proving an MDP for such random sums. To give the term MDP a precise mathematical meaning, recall that a sequence of random variables $(Y_n)_n$ with values in $\R$ (more general sets such as Polish spaces can be considered, but for our purposes $\R$ is sufficient) obeys a large deviation principle (LDP) with speed $v(n)$ ($v(n)$ is a sequence of real numbers with $v(n) \to \infty$ as $n \to \infty$) and good
rate function $I(\cdot):\R \to \R^+_0$ if
\begin{itemize}
\item $I$  has compact level sets
$N_L:=\{x\in \R: I(x) \le L\}$, for every $L \in  \R^+_0$.
\item
For every open set $G \subseteq \R$ it holds
\begin{equation}
\liminf_{n \to \infty} \frac 1 {v(n)} \log \P(Y_n \in G)\ge -\inf_{x\in G} I(x).
\end{equation}
\item
For every closed set $ A\subseteq \R$ it holds
\begin{equation}
\limsup_{n \to \infty} \frac 1 {v(n)} \log \P(Y_n \in A)\le -\inf_{x\in A} I(x).
\end{equation}
\end{itemize}

Formally, there is no distinction between an MDP and an LDP.
Usually an LDP lives on the scale of a law of large numbers, while MDPs describe the probabilities on a scale
between a law of large numbers and a central limit theorem.
An important difference between an LDP and an MDP is that typically, the rate function in
an LDP will depend on the distribution of the underlying random
variables, while an MDP inherits properties of both, the central limit behavior as well
as of the LDP. For example, one often sees the exponential decay
of moderate deviation probabilities which is typical of the large deviations. On the other hand
the rate function in an MDP quite often is ``universal'' in the sense that it only depends
on the limiting density in the central limit theorem for these variables but not on
individual characteristics of their distributions.
Often even the rate function of an MDP interpolates between the
logarithmic probabilities that can be expected from the central limit theorem and the large deviations
rate function -- even if the limit is not normal
(see e.g. \cite{Eichelsbacher/Schmock:2002}, \cite{Loewe/Merkl:2001}).
Situations where this is not the case are particularly interesting (see e.g. \cite{Eichelsbacher/Loewe:2004} or \cite{Loewe/Meiners:2012}).

Large deviations for special situations of random sums have been considered, among others in \cite{Fleischmann/Wachtel:2008} or \cite{Klueppelberg/Mikosch:1997}. In Sections 3 and 4
we will primarily be interested in their moderate deviation behavior, even though, as a byproduct, we will also obtain a large deviation result.
Another line of research is to consider estimations of cumulants  for various statistics to obtain a precise asymptotic analysis of their distributions, see \cite{SaulisStratulyavichus:1989} and references therein. In \cite{DoeringEichelsbacher:2010} is has been shown how to relate bounds on cumulants to prove an MDP. In Section 3 we will present bounds on cumulants of $S_{\nu}$ due to \cite{DeltuvieneSaulis:2007} and will show MDPs for certain random sums.
Similar to the behavior when central limit theorems are considered,
the precise form of the MDPs will crucially depend on the concentration properties of the summation variable. Finally, in Section 5, we also prove a large deviations principle for certain random sums.

{\textbf{Acknowledgement}}: We are very grateful to anonymous referee for a very careful reading of a first version of this manuscript. His comments helped to improve the correctness of the paper. 

\section{Sums of independent random variables and martingales}

In this section we will meet the replacement trick in a toy example, namely for sums of i.i.d. random variables with a moment generating function that exists in a neighbourhood of the origin. As it turns out the same trick can be applied to martingales, if the martingale differences have a locally existing moment generating function. The first of the following two theorems can already be found in \cite{Dembo/Zeitouni:LargeDeviations}, Theorem 3.7.1. Note that in this reference also the $d$-dimensional case is treated, however, to keep this exposition as simple as possible, we will restrict ourselves to the one-dimensional case. The d-dimensional case is very similar.
We want to show:

\begin{theorem}\label{iid}
Let $X_1, \ldots, X_n, \ldots$ be a sequence
of $\R$-valued i.i.d. random variables such that
\begin{equation} \label{finmomgen}
\Lambda(\lambda):=\log \E[\exp(\lambda X_1)] < \infty
\end{equation}
for all $\lambda\in (-D,D)$ and some $D>0$, and, without loss of generality $\E(X_1) = 0$ and $\V(X_1)=1$. Then for any sequence $(a_n)_n$ with $a_n \to \infty$, such that $a_n/\sqrt n \to 0$ as $n \to \infty$, the sequence
$\frac 1 {a_n \sqrt n}\sum_{i=1}^n X_i =: \frac 1 {a_n \sqrt n} S_n$ obeys an MDP with speed $a_n^2$ and rate function $I(x)=x^2/2$.
\end{theorem}

\begin{proof}
It may be instructive to recall Lindeberg's method to prove a CLT in a nutshell. There, one takes a test function $g: \R \to \R$ in the class $C_b^3(\R)$, the class
of a three times differentiable  functions with bounded derivatives, and compares
$\E g(\frac 1 {\sqrt n} S_n)$ to $\E g(\frac 1 {\sqrt n} \sum_{i=1}^n Z_i)$, where the $Z_i$ are i.i.d. standard normal random variables that are independent of the $(X_i)_i$. This comparison is done by writing
\begin{eqnarray*}
&& \biggl|\E g\bigl(\frac 1 {\sqrt n} S_n\bigr)-\E g\bigl(\frac 1 {\sqrt n} \sum_{i=1}^n Z_i\bigr)\biggr|\\
&=& \biggl|\sum_{m=1}^n \E g\bigl(\frac 1 {\sqrt n} (W_m+U_m+X_m)\bigr)-\E g\bigl(\frac 1 {\sqrt n} (W_m+U_m+Z_m)\bigr)\bigr)\biggr|,
\end{eqnarray*}
where $U_m:=\sum_{j=1}^{m-1} X_j$ and $W_m:=\sum_{j=m+1}^n Z_j$ with $U_1=W_n=0$. The summands on the right hand side are then compared by Taylor expanding $g$ and by the fact that the first two moments of the $X_i$ and $Z_i$ match. Therefore each summand is of order $C/n^{3/2}$ for any fixed $g \in C_b^3(\R)$
and therefore the right hand side can be bounded by $C/\sqrt n$. Note that $C$ depends on $g$.

In view of the G\"artner-Ellis theorem for an MDP the test functions are the logarithmic moment generating function and due to the logarithm we have to take quotients instead of differences. There are several (similar) ways to implement this idea. One is to consider for any $t \in \R$ and with the definitions of $U_m$ and $W_m$ as above
\begin{eqnarray*}
\frac{\E e^{t \frac {a_n}{\sqrt n} \sum_{i=1}^n X_i}}{\E e^{t \frac {a_n}{\sqrt n} \sum_{i=1}^n Z_i}}
&=& 1+ \frac{\E e^{t \frac {a_n}{\sqrt n} \sum_{i=1}^n X_i}-\E e^{t \frac {a_n}{\sqrt n} \sum_{i=1}^n Z_i}}{\E e^{t \frac {a_n}{\sqrt n} \sum_{i=1}^n Z_i}}\\
&=& 1+\sum_{m=1}^n \frac{\E e^{t \frac {a_n}{\sqrt n} (W_m+U_m+X_m)}-\E e^{t \frac {a_n}{\sqrt n} (W_m+U_m+Z_m)}}{\E e^{t \frac {a_n}{\sqrt n} \sum_{i=1}^n Z_i}}.
\end{eqnarray*}
 Due to the fact that the $(Z_i)_i$ are independent Gaussian random variables the expression for the denominator is
$$
\E e^{t \frac {a_n}{\sqrt n} \sum_{i=1}^n Z_i} = e^{t^2 a_n^2/2}.
$$
Using the independence of the summands and hence the independence of $W_m+U_m$ and $X_m$ as well as the independence of $W_m+U_m$ and $Z_m$, each summand in the  nominator can be represented as
$$
\E e^{t \frac{a_n}{\sqrt n}(W_m+U_m)}  \bigl( \E e^{t \frac{a_n}{\sqrt n} X_m} - \E e^{t \frac{a_n}{\sqrt n} Z_m} \bigr).
$$
Note that $\frac{a_n}{\sqrt n} \to 0$ as $n \to \infty$, and -- because $\Lambda(\lambda)< \infty$ in a ball around the origin -- $\E e^{t \frac{a_n}{\sqrt n} X_m} < \infty$ for each $t \in \R$ and for $n$ large enough. By dominated convergence, we obtain
$$
\E e^{t \frac{a_n}{\sqrt n} X_m} = 1+t \frac {a_n}{\sqrt n} \E(X_m)+ \frac{t^2} 2 \frac {a^2_n}{n} \E(X^2_m)+ {\mathcal O} \bigl( t^3 \frac {a^3_n}{n^{3/2}} \bigr)
$$
and the same estimate for $\E e^{t \frac{a_n}{\sqrt n} Z_m}$:
$$
\E e^{t \frac{a_n}{\sqrt n} Z_m} = 1+t \frac {a_n}{\sqrt n} \E(Z_m)+ \frac{t^2} 2 \frac {a^2_n}{n} \E(Z^2_m)+ {\mathcal O} \bigl( t^3 \frac {a^3_n}{n^{3/2}} \bigr).
$$
Using the fact that the first two moments of $X_i$ and $Z_i$ agree, we obtain
$$
\E e^{t \frac{a_n}{\sqrt n} X_m} - \E e^{t \frac{a_n}{\sqrt n} Z_m} ={\mathcal O} \bigl( t^3 \frac {a^3_n}{n^{3/2}} \bigr),
$$
and hence
\begin{eqnarray*}
\frac{\E e^{t \frac {a_n}{\sqrt n} \sum_{i=1}^n X_i}}{\E e^{t \frac {a_n}{\sqrt n} \sum_{i=1}^n Z_i}}
&=& 1+ \frac{{\mathcal O} \bigl( t^3 \frac{a_n^3}{n^{3/2}} \bigr) \sum_{m^=1}^n \E e^{t \frac{a_n}{\sqrt n}(W_m+U_m)}}{e^{t^2 a_n^2/2}}.
\end{eqnarray*}
Observe that by monotonicity
$$
\sum_{m^=1}^n \E e^{t \frac{a_n}{\sqrt n}(W_m+U_m)} \le n \max_{m=1, \ldots,n}  \E e^{t \frac{a_n}{\sqrt n}(W_m+U_m)} \le n \max
\{ e^{t^2 a_n^2/2}, (\E e^{\frac{t a_n}{\sqrt n} X_1})^n\}.
$$
Now again due to the fact that $X_1$ is centered, has variance 1, and $\frac{a_n}{\sqrt n}$ is going to 0
$$
(\E e^{\frac{t a_n}{\sqrt n} X_1})^n \le (1+\frac{t^2 a_n^2} n + C\frac{t^3 a_n^3}{n^{3/2}})^n \le e^{t^2 a_n^2}(1+K)
$$
for constants $C$ and $K$. Thus
$$
\frac 1 {a_n^2} \log \frac{\E e^{t \frac {a_n}{\sqrt n} \sum_{i=1}^n X_i}}{\E e^{t \frac {a_n}{\sqrt n} \sum_{i=1}^n Z_i}} \le \tilde K t^3 \frac{a_n}{\sqrt n} \to 0
$$
for some constant $\tilde K$ as $n$ goes to infinity. Interchanging the roles of $X_i$ and $Z_i$ also yields the reverse bound which eventually proves that
$$
\lim_{n \to \infty} \frac 1 {a_n^2} \log \E e^{t \frac {a_n}{\sqrt n} \sum_{i=1}^n X_i}=t^2/2
$$
which in view of the G\"artner-Ellis theorem \cite[Theorem 2.3.6 ]{Dembo/Zeitouni:LargeDeviations} is all we need, since the Legendre-Fenchel-transform
of $t^2/2$ is $t^2/2$.
\end{proof}

In view of Section \ref{sec_cum} let us remark:

\begin{remark} Condition \eqref{finmomgen} is called {\it Cram\'er's condition}. It is known that this condition is equivalent to {\it Bernstein's conditions}
that for $k \geq 3$
$$
\E (X_1^k) \leq \frac 12 k! \, H^{k-2} \E(X_1^2),
$$
where $H$ is a positive absolute constant. We will consider this condition in subsection 3.2. as well.
Moreover note that in the i.i.d. case the MDP in Theorem \ref{iid} can be verified under the less restrictive (necessary and sufficient) conditions $\E X_1 =0$ and
$$
\limsup_{n\to \infty} \frac{1}{a_n^2} \log \bigl( n \, P (|X_1| > a_n \sqrt{n}) \bigr) = -\infty,
$$
see \cite{Ledoux:1992} and \cite{Eichelsbacher/Loewe:2001}.
\end{remark}

An interesting aspect of the above proof is that it can immediately be generalized to martingales with bounded jumps. For such martingales the moderate deviations have been analyzed (even on a path level) in \cite{Dembo:1996} and one could obtain our result via contraction from there. We will use a Lindeberg approach, however.
Indeed, suppose now that $S_n$ is a discrete time martingale such that the martingale differences
$X_n:= S_n-S_{n-1}$ have conditional expectation $\E[X_n| {\mathcal F}_{n-1}]=0$ and conditional variances $\sigma^2=\E[X_n^2|\mathcal{F}_{n-1} ]=1$, where $\mathcal{F}_n = \sigma(X_1, \ldots, X_n)$ is the canonical filtration. Then we can follow the above arguments by simply replacing
the $X_i$ by their conditional versions $\E[X_i|\mathcal{F}_{n-1}]$ to obtain

\begin{theorem}
Let $(X_i)_i$ be a sequence of $\R$-valued, bounded martingale differences
such that
$\E[X_n|\mathcal{F}_{n-1}] = 0$ and $\sigma_n^2 := \E[X_n^2|\mathcal{F}_{n-1}]=1$ a.s. for all $n$.
Then for any sequence $(a_n)_n$ with $a_n \to \infty$, such that $\lim_{n \to \infty} a_n/\sqrt n = 0$, the sequence
$\frac 1 {a_n \sqrt n}\sum_{i=1}^n X_i =: \frac 1 {a_n \sqrt n} S_n$ obeys an MDP with speed $a_n^2$ and rate function $I(x)=x^2/2$.
\end{theorem}


\begin{proof}
Let $X = (X_1, \ldots, X_n)$ be as in the statement of the theorem and $Z_1, \ldots, Z_n$ be independent standard normal distributed random variables that are independent of the $(X_i)_i$. As in the proof of Theorem \ref{iid} we consider $U_m:=\sum_{j=1}^{m-1} X_j$ and $W_m:=\sum_{j=m+1}^n Z_j$ with $U_1=W_n=0$.
Along the lines of the proof of Theorem \ref{iid} we have to consider
$$
\E e^{t \frac {a_n}{\sqrt n} (W_m+U_m+X_m)}-\E e^{t \frac {a_n}{\sqrt n} (W_m+U_m+Z_m)}
$$
for any $t \in \R$. The random variable $W_m$ is normally distributed and is independent of $U_m$, $X_m$ and $Z_m$ as well as of ${\mathcal F}_{m-1}$.
As $U_m$ is ${\mathcal F}_{m-1}$-measurable, we have
\begin{eqnarray*}
\E e^{t \frac {a_n}{\sqrt n} (W_m+U_m+X_m)} &=& \E \bigl(  \E \bigl( e^{t \frac {a_n}{\sqrt n} (W_m+U_m+X_m)} | {\mathcal F}_{m-1} \bigr) \bigr)\\
& =&
 \E e^{t \frac {a_n}{\sqrt n} W_m} \E \bigl(  e^{t \frac {a_n}{\sqrt n} U_m} \E \bigl(  e^{t \frac {a_n}{\sqrt n} X_m} | {\mathcal F}_{m-1} \bigr) \bigr),
 \end{eqnarray*}
 and therefore
\begin{eqnarray*}
&& \E e^{t \frac {a_n}{\sqrt n} (W_m+U_m+X_m)}-\E e^{t \frac {a_n}{\sqrt n} (W_m+U_m+Z_m)} \\
&=& \E e^{t \frac {a_n}{\sqrt n} W_m} \E \bigl(  e^{t \frac {a_n}{\sqrt n} U_m}  \E \bigl(  e^{t \frac {a_n}{\sqrt n} X_m} | {\mathcal F}_{m-1} \bigr)
- \E \bigl(  e^{t \frac {a_n}{\sqrt n} Z_m} | {\mathcal F}_{m-1} \bigr) \bigr).
\end{eqnarray*}
As in the proof of Theorem \ref{iid} note that $\frac{a_n}{\sqrt n} \to 0$ as $n \to \infty$, and because $X_i$ are bounded, of course also $\E \bigl( e^{t \frac{a_n}{\sqrt n} X_m} | {\mathcal F}_{m-1} \bigr) < \infty$ for each $t \in \R$ and for $n$ large enough. By dominated convergence it follows a.s.
$$
\E \bigl( e^{t \frac{a_n}{\sqrt n} X_m} |  {\mathcal F}_{m-1} \bigr) = 1+t \frac {a_n}{\sqrt n} \E(X_m |  {\mathcal F}_{m-1})+ \frac{t^2} 2 \frac {a^2_n}{n} \E(X^2_m  |{\mathcal F}_{m-1} )+ {\mathcal O} \bigl( t^3 \frac {a^3_n}{n^{3/2}} \bigr).
$$
(Note that this is the part where we rely on the boundedness assumption for the $X_i$; if we had just assumed a finite moment generating function as Theorem \ref{iid}, the $\mathcal{O}-term$ would be random variable depending on the rest of the martingale).
By the same Taylor-argument for $\E \bigl( e^{t \frac{a_n}{\sqrt n} Z_m} \bigr)$ we arrive at
$$
\E \bigl( e^{t \frac{a_n}{\sqrt n} X_m} |  {\mathcal F}_{m-1} \bigr) = 1+t \frac {a_n}{\sqrt n} \E(Z_m)+ \frac{t^2} 2 \frac {a^2_n}{n} \E(Z^2_m)+ {\mathcal O} \bigl( t^3 \frac {a^3_n}{n^{3/2}} \bigr).
$$
With  $\E(X_m |  {\mathcal F}_{m-1} )=0 = \E (Z_m)$ and $\E(X_m^2 |  {\mathcal F}_{m-1})=1 = \E (Z_m^2)$ we obtain
$$
\E \bigl(  e^{t \frac {a_n}{\sqrt n} X_m} | {\mathcal F}_{m-1} \bigr)
- \E \bigl(  e^{t \frac {a_n}{\sqrt n} Z_m} | {\mathcal F}_{m-1} \bigr) ={\mathcal O} \bigl( t^3 \frac {a^3_n}{n^{3/2}} \bigr).
$$
Finally we have to bound $\sum_{m=1}^n \E e^{t \frac {a_n}{\sqrt n} (W_m +U_m)}$. We observe that
$$
\sum_{m^=1}^n \E e^{t \frac{a_n}{\sqrt n}(W_m+U_m)} \le n \max_{m=1, \ldots,n}  \E e^{t \frac{a_n}{\sqrt n}(W_m+U_m)}=
n \max_{m=1, \ldots,n}  \E e^{t \frac{a_n}{\sqrt n}W_m}\E e^{t \frac{a_n}{\sqrt n}U_m}.
$$
Now on one hand by Gaussian integration
$$
\E e^{t \frac{a_n}{\sqrt n}W_m}= \exp\left(\frac{a_n^2 t^2}{2}\frac{n-m}n\right).
$$
On the other hand, preparing the same expansion technique as before we note that
\begin{eqnarray*}
\E e^{t \frac{a_n}{\sqrt n}U_m}&=&\E \bigl( e^{\frac{t a_n}{\sqrt n} \sum_{j=1}^{m-1}X_j} \bigr) = \E \bigl( \E  \bigl(e^{\frac{t a_n}{\sqrt n} \sum_{j=1}^{m-1}X_j} | {\mathcal F}_{m-2} \bigl) \bigr)\\
& = &
\E \biggl( \prod_{j=1}^{m-2} e^{\frac{t a_n}{\sqrt n} X_j} \, \E \bigl( e^{\frac{t a_n}{\sqrt n} X_{m-1}} | {\mathcal F}_{m-2} \bigl) \biggr)
\end{eqnarray*}
Now by our assumption we have
\begin{eqnarray*}
\E \bigl( e^{\frac{t a_n}{\sqrt n} X_{m-1}} | {\mathcal F}_{m-2} \bigl) & = &1 + t \frac{a_n}{\sqrt n} \E (X_{m-1} | {\mathcal F}_{m-2}) +
\frac{t^2}{2} \frac{a_n^2}{n} \E (X_{m-1}^2 | {\mathcal F}_{m-2}) + {\mathcal O} \bigl( t^3 \frac {a^3_n}{n^{3/2}} \bigr)\\
& = &1 + \frac{t^2}{2} \frac{a_n^2}{n}
+ {\mathcal O} \bigl( t^3 \frac {a^3_n}{n^{3/2}} \bigr),
\end{eqnarray*}
yielding by iteration
$$
 \E \bigl( e^{\frac{t a_n}{\sqrt n} \sum_{j=1}^{m-1}X_j} \bigr) \leq \biggl( 1 + \frac{t^2}{2} \frac{a_n^2}{n}
+ {\mathcal O} \bigl( t^3 \frac {a^3_n}{n^{3/2}} \bigr) \biggr)^{m-1}.
$$
Putting this together with the Gaussian part we see that
$$
\E e^{t \frac {a_n}{\sqrt n} (W_m +U_m)}\le \exp\left(\frac{t^2}2 a_n^2\right)(1+K)
$$
for some constant $K$ (actually for any $K$, if $n$ is large enough).
Now along the same arguments as in the proof of Theorem \ref{iid} we end up with the desired statement.
\end{proof}

\section{Random sums}
We now turn to the core of this paper, the derivation of MDPs for random sums.
In general random sums of the following type are considered. Let $(X_i)_i$ be independent random variables and denote $S_k := \sum_{j=1}^k X_j$.
Let $(a_k)_k$ and $(b_k)_k$ be sequences of real numbers with $b_k >0$ for $k \geq 1$ and denote
\begin{equation} \label{Yk}
Y_k = \frac{S_k - a_k}{b_k}.
\end{equation}
Let $(\nu_k)_k$ be positive integer valued random variables independent of the sequence $(X_i)_i$ for every $k \geq 1$. The aim is to study the asymptotic behaviour of the random variables
$$
Z_k = \frac{S_{\nu_k} - c_k}{d_k},
$$
where $(c_k)_k$ and $(d_k)_k$ are sequences of real numbers, $d_k >0$ for $k \geq 1$.  Assume that
\begin{equation} \label{conv}
(Y_k)_k \quad \text{converges weakly to some random variable} \quad Y \quad \text{as} \quad k \to \infty.
\end{equation}
Theorem 3.1.2 in \cite{Gnedenko/Korolev:1996} presents possible approximating laws:

\begin{theorem}
Let sequences of real numbers $(a_k)_k, (b_k)_k, (c_k)_k$ and $(d_k)_k$  be such that $b_k \to \infty, d_k \to \infty$ for $k \to \infty$ and condition \eqref{conv} is satisfied.
Let
$$
\biggl( \frac{b_{\nu_k}}{d_k}, \frac{a_{\nu_k} - c_k}{d_k} \biggr) \Rightarrow (U,V)
$$
as $k \to \infty$ for some random variables $U$ and $V$. Then
$$
\lim_{k \to \infty} P(Z_k < x) = \E H\bigl( \frac{x-V}{U} \bigr),
 $$
 where $H$ denotes the distribution function of $Y$.
\end{theorem}
\medskip

\subsection{The Gaussian case with Lindeberg's method}
In this subsection we consider the case where random sums have a normal distribution in limit. For this case the following conditions are necessary and sufficient
for a central limit theorem (see \cite[Theorem 3.3.3]{Gnedenko/Korolev:1996}):

\begin{theorem} \label{nor}
Assume that $(\nu_k)_k$ converges to $\infty$ in probability as $k \to \infty$ and assume \eqref{conv}.
$(Z_k)_k$ converges weakly to a standard normal distributed random for some sequences of positive numbers $(d_k)_k$ with $d_k \to \infty$ as $k \to \infty$ {\it if and only if} the following conditions are fulfilled:
$$P(Y<x) = \Phi(x) \qquad \mbox{and } \quad \frac{b_{\nu_k}}{d_k}\to 1$$ weakly as $k \to \infty$. Here $\Phi$ denotes the distribution function of a standard normal distributed random variable.
\end{theorem}

In this paper we consider only the case of i.i.d. random variables $(X_i)_i$ with mean and finite positive variance
$$
a=\E(X_1), \qquad  0 < c^2 = \V(X_1) < \infty.
$$
Hence in \eqref{Yk} we choose $a_k= a \, k$ and $b_k = c \, \sqrt{k} $.

In addition, the mean and the variance of $\nu_k$ are denoted by
$$
\mu_k = \E(\nu_k), \qquad  \gamma_k^2 = \V(\nu_k).
$$
Therefore we choose $c_k = a \, \mu_k$ and $d_k^2 = c^2 \mu_k
+ a^2 \gamma_k^2$ applying Blackwell-Girshick.

Hence for the choice of i.i.d. random variables $(X_i)_i$ the second condition in Theorem \ref{nor} reads as
$\frac{c \sqrt{\nu_k}}{d_k} \to 1$ weakly. Finally we restrict ourselves to the standardized version $a=0$ and $c^2=1$, which leads to $c_k=0$ and $d_k^2= \mu_k$.

We will therefore assume that $\nu_k \to\infty$ in probability and
\begin{equation}\label{Gaussiancase}
\frac{\nu_k}{\mu_k} \to 1 \quad \mbox{in probability, as } k \to \infty.
\end{equation}
Thus \eqref{Gaussiancase} is a necessary and sufficient condition for the convergence in distribution of
$$
Z_{k,0}:= \frac 1{\sqrt {\mu_k}} \sum_{i=1}^{\nu_k} X_i
$$
towards a standard normal distribution.

Naturally, if we ask for moderate deviations for rescaled sums of such type, we need also stronger concentration properties of $\nu_k$. Let us introduce the quantities of interest
\begin{equation}\label{Zka}
Z_{k,\alpha}:= \frac 1{\mu_k^{1/2+\alpha}} \sum_{i=1}^{\nu_k} X_i
\end{equation}
for $0< \alpha <1/2$. To study their moderate deviations we will assume the following concentration properties of $(\nu_k)_k$.

\begin{ass}\label{firstass}
Assume that the sequence of random variables $\bigl( \nu_k/\mu_k \bigr)_k$ obeys an LDP with speed $\mu_k^\beta$, for some $\beta > 2\alpha$ and rate function $I$ that has a unique minimum $I(x)=0$ for $x=1$.
\end{ass}

\noindent
Under this assumption we are able to show

\begin{theorem}\label{randsumgauss}
Suppose that the $(X_i)_i$ are i.i.d. centred random variables with variance one and that \eqref{finmomgen} is satisfied. Moreover assume that the random summation index $(\nu_k)_k$ satisfies Assumption \ref{firstass}.
Then, for $0< \alpha <1/2$, the sequence of random variables $(Z_{k,\alpha})_k$ satisfies and MDP with speed $\mu_k^{2\a}$ and rate function $J(x)=x^2/2$.
\end{theorem}

\begin{remark}
Note that this results fits into the folklore that for random variables that satisfy an ordinary Central Limit Theorem that rate function in an MDP is always $x^2/2$.
\end{remark}

\begin{proof}
We will again apply the G\"artner-Ellis-Theorem. To this end we will start with a sequence $(X_i)_i$ of i.i.d. standard normal random variables.
We obtain
\begin{eqnarray*}
\E\bigl( e^{t \mu_k^{2\alpha} Z_{k,\a}}\bigr)&=& \E \bigl( \E[e^{t \mu_k^{\alpha-\frac 12 } \sum_{i=1}^{\nu_k} X_i}|\nu_k] \bigr)\\
&=&\E\big( e^{\frac{t^2}2 \mu_k^{2\alpha-1}\nu_k}\big)\\
&=& \sum_{n \ge 0}\P(\nu_k=n) e^{\frac{n t^2}2 \mu_k^{2\alpha-1}}\\
&=& \sum_{x \ge 0}\P\left(\frac{\nu_k}{\mu_k}=x\right) e^{\frac{x t^2}2 \mu_k^{2\alpha}},\\
\end{eqnarray*}
where the last sum is, of course, taken over all $x$ in the image of the random variable $\frac{\nu_k}{\mu_k}$. We wish to show that
$$
\lim_{k \to \infty} \frac 1{\mu_k^{2 \a}}\log \sum_{x \ge 0}\P\left(\frac{\nu_k}{\mu_k}=x\right) e^{\frac{x t^2}2 \mu_k^{2\alpha}}= t^2/2.
$$
To this end we proceed in a way, that is well established in the large deviation literature.
For the lower bound just note that for any $\vep>0$
\begin{eqnarray*}
\sum_{x \ge 0}\P\left(\frac{\nu_k}{\mu_k}=x\right) e^{\frac{x t^2}2 \mu_k^{2\alpha}} & \ge &
\sum_{x \in (1-\vep, 1+\vep)}\P\left(\frac{\nu_k}{\mu_k}=x\right) e^{\frac{x t^2}2 \mu_k^{2\alpha}} \\
&\ge& e^{\frac{(1-\vep) t^2}2 \mu_k^{2\alpha}} \P\left(\frac{\nu_k}{\mu_k}\in (1-\vep, 1+\vep)\right).
\end{eqnarray*}
Now since we assumed that $\frac{\nu_k}{\mu_k}$ converges to 1 in probability, we immediately have that 
$$
\lim \frac 1{\mu_k^{2 \a}}\log \P\left(\frac{\nu_k}{\mu_k}=1\right) =0.
$$
This readily implies 
$$
\liminf_{k \to \infty} \frac 1{\mu_k^{2 \a}}\log\E \bigl( e^{t \mu_k^{2\alpha} Z_{k,\a}} \bigr) \ge (1-\vep)t^2/2.
$$
Since this is true for all $\vep>0$, we obtain
$$
\liminf_{k \to \infty} \frac 1{\mu_k^{2 \a}}\log\E \bigl( e^{t \mu_k^{2\alpha} Z_{k,\a}} \bigr) \ge t^2/2.
$$
On the other hand, for $\vep >0$ we denote by $B_\vep(1)$ the ball of radius $\vep$ centered in 1 in the set of images of $\frac{\nu_k}{\mu_k}$ (denoted by $\mathrm{Im}[\frac{\nu_k}{\mu_k}]$). Then we obtain
\begin{eqnarray}\label{zweisumm}
&& \sum_{x \ge 0} \P\left(\frac{\nu_k}{\mu_k}=x\right) e^{\frac{x t^2}2 \mu_k^{2\alpha}} \nonumber \\
&=&
\sum_{x \in B_\vep(1)}\P\left(\frac{\nu_k}{\mu_k}=x\right) e^{\frac{x t^2}2 \mu_k^{2\alpha}}+
\sum_{x \in B_\vep^c(1)}\P\left(\frac{\nu_k}{\mu_k}=x\right) e^{\frac{x t^2}2 \mu_k^{2\alpha}}.
\end{eqnarray}
As
$$
I_\vep:= \inf_{x \in B_\vep^c(1)} I(x) >0
$$
there exists $\delta >0$, such that even $I_\vep-\delta >0$.
On the other hand, by the LDP assumption on $(\nu_k/\mu_k)_k$ for $k$ large enough we have
$$
\sum_{x \in B_\vep^c(1)}\P\left(\frac{\nu_k}{\mu_k}=x\right) e^{\frac{x t^2}2 \mu_k^{2\alpha}} \ge
\sum_{x \in B_\vep^c(1)} e^{-\mu_k^\b (I_\vep -\delta -f_k(x))},
$$
where
$$
f_k(x)=x \frac{t^2}2 \mu_k^{2 \alpha-\b}.
$$
As $\lim_{k \to \infty}f_k(x)= 0 $ for every $x$ we can apply Lemma 2.2. in \cite{Varadhan:StFlour} to obtain
that the second summand on the right in \eqref{zweisumm} behaves like  $e^{-\mu_k^\b (I_\vep -\delta)}$ which implies that
$$
\lim_{k \to \infty} \frac 1{\mu_k^{2 \a}}\log \sum_{x \in B_\vep^c(1)}\P\left(\frac{\nu_k}{\mu_k}=x\right) e^{\frac{x t^2}2 \mu_k^{2\alpha}}=-\infty.
$$
For the first summand on the right in \eqref{zweisumm} observe that
$$
\sum_{x \in B_\vep(1)}\P\left(\frac{\nu_k}{\mu_k}=x\right) e^{\frac{x t^2}2 \mu_k^{2\alpha}} \le e^{(1+\vep) \frac{t^2}2 \mu_k^{2\a}}.
$$
With
$$
T_1 :=\limsup_k\frac 1{\mu_k^{2 \a}}\log \sum_{x \in B_\vep(1)}\P\left(\frac{\nu_k}{\mu_k}=x\right) e^{\frac{x t^2}2 \mu_k^{2\alpha}}
$$
and
$$
T_2 := \limsup_k \frac 1{\mu_k^{2 \a}}\log \sum_{x \in B_\vep^c(1)}\P\left(\frac{\nu_k}{\mu_k}=x\right) e^{\frac{x t^2}2 \mu_k^{2\alpha}}
$$
we obtain altogether that
$$
\lim_{k \to \infty} \frac 1{\mu_k^{2 \a}}\log \sum_{x}\P\left(\frac{\nu_k}{\mu_k}=x\right) e^{\frac{x t^2}2 \mu_k^{2\alpha}}
\le \max \{T_1, T_2 \} \le \max\{(1+\vep) t^2/2, -\infty\},
$$
and therefore, by letting $\vep \to 0$
$$
\lim_{k \to \infty} \frac 1{\mu_k^{2 \a}}\log \sum_{x}\P\left(\frac{\nu_k}{\mu_k}=x\right) e^{\frac{x t^2}2 \mu_k^{2\alpha}}
\le t^2/2.
$$
Together with the lower bound above this completes the analysis of the situation where the $(X_i)_i$ are standard normally distributed.

For general random variables $(X_i)_i$ that satisfy the conditions of the theorem we employ our replacement trick again. Note that for the G\"artner-Ellis theorem we need to analyze
$ \E \bigl( e^{t \mu_k^{\alpha-\frac 12 } \sum_{i=1}^{n} X_i} \bigr)$.

Following the lines of the proof of Theorem \ref{iid} one shows that
for i.i.d. standard Gaussian random variables $(Y_i)_i$ and $(X_i)_i$ as in the statement of our theorem one sees that there are constants $C_1, C_2>0$ such for each fixed $n$
$$
\exp(-C_1 t^3 \mu_k^{3 \a -\frac 32 }n)\le \frac{ \E \bigl( e^{t \mu_k^{\alpha-\frac 12 } \sum_{i=1}^{n} X_i} \bigr)}{ \E \bigl(e^{t \mu_k^{\alpha-\frac 12 } \sum_{i=1}^{n} Y_i} \bigr)}\le \exp(C_2 t^3 \mu_k^{3 \a -\frac 32 }n).
$$
Therefore
\begin{eqnarray*}
\E \bigl( e^{t \mu_k^{2\alpha} Z_{k,\a}} \bigr)&=&\E \bigl( e^{t \mu_k^{\alpha-\frac 12 } \sum_{i=1}^{\nu_k} X_i} \bigr)\\
&=& \sum_{n \ge 0}\P(\nu_k=n) \E \bigl( e^{t \mu_k^{\alpha-\frac 12}\sum_{i=1}^n X_i} \bigr)\\
&\le& \sum_{n \ge 0}\P(\nu_k=n) \E \bigl(e^{t \mu_k^{\alpha-\frac 12}\sum_{i=1}^n Y_i} \bigr)  e^{C_2 t^3 \mu_k^{3 \a -\frac 32 }n}
\end{eqnarray*}
as well as
$$
\E \bigl( e^{t \mu_k^{2\alpha} Z_{k,\a}} \bigr) \ge  \sum_{n \ge 0}\P(\nu_k=n) \E \bigl(e^{t \mu_k^{\alpha-\frac 12}\sum_{i=1}^n Y_i} \bigr) e^{-C_1 t^3 \mu_k^{3 \a -\frac 32 }n}.
$$
Also recall that $\E \bigl( e^{t \mu_k^{\alpha-\frac 12}\sum_{i=1}^n Y_i} \bigr)= e^{n\frac {t^2}2 \mu_k^{2\alpha-1}}$.
Again we reparamatrize the sums on the right hand sides of the above displays as
$$
\sum_{n \ge 0}\P(\nu_k=n) \E \bigl( e^{t \mu_k^{\alpha-\frac 12}\sum_{i=1}^n Y_i} \bigr) e^{C_2 t^3 \mu_k^{3 \a -\frac 32 }n} =
\sum_{x \ge 0}\P\left(\frac{\nu_k}{\mu_k}=x\right) e^{\frac{t^2}2 \mu_k^{2\alpha}x} e^{C_2 t^3 \mu_k^{3 \a -\frac 12 }x}
 $$
 and
 $$
\sum_{n \ge 0}\P(\nu_k=n) \E \bigl(e^{t \mu_k^{\alpha-\frac 12}\sum_{i=1}^n Y_i} \bigr) e^{-C_1 t^3 \mu_k^{3 \a -\frac 32 }n}=
\sum_{x \ge 0}\P\left(\frac{\nu_k}{\mu_k}=x\right) e^{\frac{t^2}2 \mu_k^{2\alpha}x} e^{-C_1 t^3 \mu_k^{3 \a -\frac 12 }x},
$$
respectively. As above one shows that asymptotically for any $\vep>0$ the probabilities get concentrated on the ball $B_\vep (1)$:
\begin{eqnarray*}
&&\lim_{k \to \infty}\frac 1{\mu_k^{2\a}}\log \sum_{x \ge 0}\P\left(\frac{\nu_k}{\mu_k}=x\right) e^{\frac{t^2}2 \mu_k^{2\alpha}x+C_2 t^3 \mu_k^{3 \a -\frac 12 }x}\\
&&\hskip 3cm= \lim_{k \to \infty}\frac 1{\mu_k^{2\a}}\log \sum_{x \in B_\vep(1)}\P\left(\frac{\nu_k}{\mu_k}=x\right) e^{\frac{t^2}2 \mu_k^{2\alpha}x+C_2 t^3 \mu_k^{3 \a -\frac 12 }x}
\end{eqnarray*}
as well as
\begin{eqnarray*}
&&\lim_{k \to \infty}\frac 1{\mu_k^{2\a}}\log \sum_{x \ge 0}\P\left(\frac{\nu_k}{\mu_k}=x\right) e^{\frac{t^2}2 \mu_k^{2\alpha}x-C_1 t^3 \mu_k^{3 \a -\frac 12 }x}\\
&&\hskip 3cm= \lim_{k \to \infty}\frac 1{\mu_k^{2\a}}\log \sum_{x \in B_\vep(1)}\P\left(\frac{\nu_k}{\mu_k}=x\right) e^{\frac{t^2}2 \mu_k^{2\alpha}x-C_1 t^3 \mu_k^{3 \a -\frac 12 }x}.
\end{eqnarray*}
On the other hand by what we learned above for the situation with Gaussian summands
\begin{eqnarray*}
&&\limsup_{k \to \infty}\frac 1{\mu_k^{2\a}}\log \sum_{x \in B_\vep(1)}\P\left(\frac{\nu_k}{\mu_k}=x\right) e^{\frac{t^2}2 \mu_k^{2\alpha}x+C_2 t^3 \mu_k^{3 \a -\frac 12 }x}\\
&&\qquad \le \limsup_{k \to \infty} t^2/2+C_2t^3\mu_k^{\a -\frac 12 }(1+\vep)= t^2/2
\end{eqnarray*}
and
\begin{eqnarray*}
&& \liminf_{k \to \infty}\frac 1{\mu_k^{2\a}}\log \sum_{x \in B_\vep(1)}\P\left(\frac{\nu_k}{\mu_k}=x\right) e^{\frac{t^2}2 \mu_k^{2\alpha}x+C_2 t^3 \mu_k^{3 \a -\frac 12 }x}\\
&&\qquad \ge \liminf_{k \to \infty} t^2/2+C_2t^3\mu_k^{\a -\frac 12 }(1-\vep)= t^2/2.
\end{eqnarray*}
And in the same way:
$$
\lim_{k \to \infty}\frac 1{\mu_k^{2\a}}\log \sum_{x \in B_\vep(1)}\P\left(\frac{\nu_k}{\mu_k}=x\right) e^{\frac{t^2}2 \mu_k^{2\alpha}x-C_1 t^3 \mu_k^{3 \a -\frac 12 }x}\to t^2/2.
$$
In view of the G\"artner-Ellis theorem \cite[Theorem 2.3.6 ]{Dembo/Zeitouni:LargeDeviations} this finishes the proof.
\end{proof}

\begin{remark}
Choosing standardized random variables $X_i$ ($a=0$ and $c^2=1$), the proof of Theorem \ref{randsumgauss} is less involved. But along the lines of the proof we would be able to prove a MDP for the {\it non-randomly centered} random sums. Here we have to assume that the sequence
$$
\biggl( \frac{c^2 \nu_k}{c^2 \mu_k+ a^2 \gamma_k^2} \biggr)_k
$$
has to satisfy an LDP with speed $(c^2 \mu_k + a^2 \gamma_k^2)^{\beta}$ with some $\beta > 2 \alpha$ and a rate function with a unique minimum at $x=1$.
Then we would observe that
$$
\frac{1}{(c^2 \mu_k+ a^2 \gamma_k^2)^{\frac 12 + \alpha}} \biggl( \sum_{i=1}^{\nu_k} X_i - a \mu_k \biggr)
$$
as a sequence in $k$ satisfies an MDP with speed $\mu_k^{2 \alpha}$ and rate $x^2/2$ for any fixed $0 < \alpha < \frac 12$.
\end{remark}

\noindent
One of the most relevant examples is the following:
\begin{example}
Suppose that the random variables $(X_i)$ satisfy the assumptions of Theorem \ref{randsumgauss} and that the summation index $\nu_\lambda$ is Poisson distributed with parameter $\lambda >0$ (where we use $\lambda$ instead of $k$ to keep the conventional notation for the parameter of a Poisson distribution). Then, as $\lambda \to \infty$, for any $0<\a<\frac 12$  the random sums $Z_{\lambda,\a} := \frac 1{\lambda^{1/2+\alpha}}\sum_{i=1}^{\nu_\lambda} X_i$ obey an MDP with speed
$\lambda^{2\a}$ and rate function $I(x)=x^2/2$.

Indeed, since $\E \nu_\lambda =\lambda$ all we need to check is that $\nu_\lambda/\lambda$ obey an LDP with a speed faster than $\lambda^{2\alpha}$. This is, however, an immediate consequence of Cram\'er's theorem \cite{Dembo/Zeitouni:LargeDeviations}, Theorem 2.1.24, since (for integer $\lambda$) $\nu_\lambda$ is equal in distribution to a sum of $\lambda$ Poisson variables with parameter 1. Alternatively, the desired property can be computed directly.
\end{example}
\medskip

\subsection{The Gaussian case via cumulants}\label{sec_cum}

Since the late seventies estimations of {\it cumulants} have been studied to investigate a more precise asymptotic analysis of the distribution of certain statistics, see e.g. \cite{SaulisStratulyavichus:1989} and references therein. In \cite{DoeringEichelsbacher:2010} moderate deviation principles are established for a rather general class of random variables fulfilling certain bounds of the cumulants.
Let $X$ be a real-valued random variable with existing absolute moments. Then with $i =\sqrt{-1}$
$$
\left. \Gamma_j := \Gamma_j(X) :=(-i)^j \frac{d^j}{dt^j} \log \E\bigl[e^{i t X}\bigr] \right|_{t=0}
$$
exists for all $j\in\mathbb N$ and the term is called the {\it $j$th cumulant}
(also called semi-invariant) of $X$. The following main theorem was proved in \cite{DoeringEichelsbacher:2010}.

\begin{theorem}\label{hanna}
For any $n \in {\Bbb N}$, let $Z_n$ be a centered random variable with variance one and existing
absolute moments, which satisfies
\begin{equation}\label{eqcumulants}
\bigl| \Gamma_j(Z_n) \bigr| \leq (j!)^{1+\gamma} / \Delta_n^{j-2}
\quad\text{for all } j=3,4, \dots
\end{equation}
for fixed $\gamma\geq 0$ and $\Delta_n>0$.
Let the sequence $(a_n)_{n \geq 1}$ of real numbers grow to infinity, but slow enough such that
$$
\frac{a_n}{\Delta_n^{1/(1+2\gamma)}} \stackrel{n\to\infty}{\longrightarrow} 0
$$
holds.
Then the moderate deviation principle for
$\bigl(\frac{1}{a_n} Z_n\bigr)_n$
with speed $a_n^2$ and rate function $I(x)=\frac{x^2}{2}$
holds true.
\end{theorem}

In \cite{DeltuvieneSaulis:2007} upper estimates for cumulants of the standardized random sum in the case where the i.i.d. random variables $(X_i)_i$
satisfy Bernstein's condition and the random variables $(\nu_k)_k$ satisfy an additional assumption for its cumulants. The result reads as follows.
Assume that $\E(X_1)=0$ and $c^2 = \V(X_1)$ and assume that there exist nonnegative numbers $K_1, K_2$ such that
\begin{equation} \label{cum1}
|\E(X_1)^j| \le j! K_1^{j-2} c^2\qquad j=3,4, \ldots
\end{equation}
and that 
\begin{equation} \label{cum2}
|\Gamma_j(\nu_k)| \leq j! K_2^{j-1} \mu_k \qquad j=1,2,\ldots, k \geq 1,
\end{equation}
where again $\mu_k = \E(\nu_k)$. 
Then in \cite{DeltuvieneSaulis:2007} it was proved that the cumulants of $(Z_{k,0})_k$ can be bounded
as following:
$$
|\Gamma_j(Z_{k,0})| \leq  \frac{j!}{\Delta_k^{j-2}} \qquad j= 3,4, \ldots.
$$
with $\Delta_k := C \sqrt{\mu_k}$ and $C$ is an explicit constant depending on $c^2, K_1, K_2$. See also \cite{Kas2013} and \cite[Lemma 2.2]{KasDiss},
where the case of weighted random sums was considered.

Applying Theorem \ref{hanna} the result in Theorem \ref{randsumgauss} follows, that is an MDP
for the sequence of random variables $(Z_{k,\alpha})_k$ holds, under conditions \eqref{cum1} and \eqref{cum2}.

Here condition \eqref{cum1}
is a Bernstein type condition for the distribution $X_1$ which is equivalent to Cram\'er's condition \eqref{finmomgen}. The second condition \eqref{cum2}
is a concentration property of $(\nu_k)_k$, which seems to be different from Assumption \ref{firstass}. With \eqref{cum2} we have that for $k$ large enough
$$
\bigg| \Gamma_j \bigl( \frac{\nu_k}{\sqrt{\mu_k}} \bigr) \bigg| \leq j!  \frac{K_2^{j-1}}{\mu_k^{j/2-1}}  \leq j! \frac{1}{\Delta_k^{j-2}} \qquad j=2, 3, \ldots
$$
with $\Delta_k := C \sqrt{\mu_k}$ and $C$ is an explicit constant depending on $K_2$.
Note that as a consequence of Theorem \ref{hanna}
we obtain that assumption \eqref{cum2} implies, that the sequence
$$\biggl( \frac{\nu_k - \mu_k}{\mu_k^{\alpha + \frac 12}} \biggr)_k
$$
satisfies an MDP for any $0< \alpha < 1/2$  with speed $\mu_k^{2 \alpha}$ and Gaussian rate function $x^2/2$, which is indeed different form the LDP assumption
in Assumption \ref{firstass}.

\subsection{Strongly Fluctuating Summation Variables}
In this section we will analyse the moderate deviations behaviour of random sums, where the summation variable does not have the strong concentration properties required in Assumption \ref{firstass}, while for the $(X_i)_i$ we will still assume that they are i.i.d. centred random variables with variance 1, such that \eqref{finmomgen} holds. Naturally, the fluctuations of a random sum on a moderate deviations scale is then fed by two sources: the fluctuation of the $(X_i)_i$ on the one hand, and the fluctuations of the summation index. The latter could, in principle, however be arbitrarily strong. To get some control over it, we have to make an assumption:

\begin{ass}\label{secondass}
Again let $\mu_k$ be the expectation of $\nu_k$ and $0<\a <\frac 12$ be fixed.
Assume that $\mu_k \to \infty$ as $k\to \infty$ and there exists $0<\gamma\le 2\a$ such that
the sequence of random variables $\bigl( \nu_k/\mu^{1+2\a-\gamma }_k\bigr)_k$ obey an LDP with speed $\mu_k^\gamma$, and a convex rate function $I$ such that
\begin{equation}\label{decay_rate}
\lim_{y \to \infty} I(y)=\infty.
\end{equation}
Alternatively, by setting $1+2\a-\gamma=\beta$ we can assume that there exists $1\le \beta < 1+ 2\a$, such that the sequence of random variables $\bigl( \nu_k/\mu^{\b}_k\bigr)_k$ obey an LDP with speed $\mu_k^{1+ 2\a -\beta}$, and a convex rate function $I$ such
that \eqref{decay_rate} is satisfied.
\end{ass}

Now we are ready to formulate our second theorem on the moderate deviations for random sums.
\begin{theorem}\label{randsum-nongauss}
Again suppose that the $(X_i)_i$ are i.i.d. centered random variables with variance one and that \eqref{finmomgen} is satisfied. Moreover assume that the random summation index satisfies Assumption \ref{secondass}. Then, for $0< \alpha <1/2$, the sequence of random variables $Z_{k,\alpha}$ defined as in \eqref{Zka} satisfies an MDP with speed $\mu_k^\gamma$ and rate function
$$
J(y)=\inf\{y^2/2s+I(s):s \in \R^+\}.
$$
\end{theorem}
\begin{remark}
Note that the behavior of $(Z_{k,\a})_k$ under Assumption \ref{secondass} is essentially different from the behavior in the previous section, which e.g. can already be seen from different speed we obtain. Moreover, even if we set
$\gamma= 2\a$ in Assumption \ref{secondass} (which formally brings us into the realm of Assumption \ref{firstass}) we reobtain the speed of Theorem \ref{randsumgauss} but it is not obvious that we have the same rate function.
\end{remark}

\begin{proof}
In the spirit of the replacement trick we start with the case, where the $(X_i)_i$ are i.i.d. standard Gaussian random variables.
We will first estimate
$\P(Z_{k,\a} \ge t)$ for $t \in \R$.
Clearly,
\begin{eqnarray}\label{zerlegung}
\P[Z_{k,\a} \ge t]
&=&\P \bigl( \{Z_{k,\a} \ge t\}\cap\{\nu_k \le c \mu_k^\beta\} \bigr) \nonumber \\
&+&\P \bigl( \{Z_{k,\a} \ge t \}\cap\{ c \mu_k^\beta \le \nu_k \le C \mu_k^\beta\} \bigr)\\
&+&\P\bigl( \{Z_{k,\a} \ge t \} \cap\{ \nu_k > C \mu_k^\beta\} \bigr)\nonumber
\end{eqnarray}
for some constants $1>c>0$ (small), $C>1$ (large) and again with $1+2\a-\gamma=\beta$. Not unexpectedly the main contribution will come from the middle summand on the right hand side in \eqref{zerlegung}.
Indeed, for the first summand we obtain
\begin{equation}\label{lower_tail}
\P \bigl( \{Z_{k,\a} \ge t\}\cap\{\nu_k \le c \mu_k^\beta\} \bigr) \le \P\left(\bigcup_{n=1}^{c \mu_k^\beta} \left\{\sum_{i=1}^n X_i \ge t \mu^{\a+\frac 12}_k\right\}\right) \le c\mu_k^\beta e^{-\frac {t^2} {2c} \mu_k^\gamma}.
\end{equation}
In fact, by a standard estimate for Gaussian random variables, the probabilities $\P\left(\sum_{i=1}^n X_i \ge t \mu^{\a+\frac 12}_k\right)$ are bounded from above by $e^{-\frac {t^2} {2n} \mu_k^{2\alpha+1}}$ which is increasing in $n$. The right hand side in \eqref{lower_tail} is this obtained by a simple union bound and the plugging in the relation between $\beta$ and $\gamma$. On the other hand Assumption \ref{secondass} provides us with the estimate
$$
P \bigl( \{Z_{k,\a} \ge t\} \cap\{ \nu_k > C \mu_k^\beta\} \bigr)  \le e^{-\mu_k^\gamma (I(C)+\vep)}
$$
for the third summand, for every $\vep>0$, if only $k$ is large enough. This is an immediate consequence of the assumed LDP for $\bigl( \nu_k / \mu_k^\beta \bigr)_k$ and the convexity of $I$. Therefore, altogether we see that
\begin{eqnarray*}
&&\P \bigl( \{Z_{k,\a} \ge t\} \cap\{ c \mu_k^\beta \le \nu_k \le C \mu_k^\beta\} \bigr) \\&&\hskip-1cm\le \P \bigl( Z_{k,\a} \ge t \bigr) \\ &&\hskip-1cm \le
\P \bigl( \{Z_{k,\a} \ge t\} \cap\{ c \mu_k^\beta \le \nu_k \le C \mu_k^\beta\} \bigr) +c\mu_k^\beta e^{-\frac {t^2}{2c} \mu_k^\gamma}+e^{-\mu_k^\gamma (I(C)+\vep)}.
\end{eqnarray*}
Now for the central term observe that
\begin{eqnarray*}
\P \bigl( \{Z_{k,\a} \ge t\} \cap\{ c \mu_k^\beta \le \nu_k \le C \mu_k^\beta\} \bigr)&=&\sum_{n=c \mu_k^\beta}^{C \mu_k^\beta} \P\left(\sum_{i=1}^{n} X_i \ge t \mu_k^{\a +\frac 12}\right)\P(\nu_k=n)\\
&=& \sum_{s=c}^{C} \P\left(\sum_{i=1}^{s \mu_k^\b} X_i \ge t \mu_k^{\a +\frac 12}\right)\P\left(\frac{\nu_k}{\mu_k^\beta}=s\right)
\end{eqnarray*}
where again the first sum is over all $s$ in the image of $\frac{\nu_k}{\mu_k^\beta}$.
Therefore, for all $0<c<C<\infty$
\begin{eqnarray*}
\max_{s \in [c, C]}\P\left(\sum_{i=1}^{s \mu_k^\b} X_i \ge t \mu_k^{\a +\frac 12}\right)\P\left(\frac{\nu_k}{\mu_k^\beta}=s\right)&\le& \P \bigl(\{Z_{k,\a} \ge t \}\cap\{ c \mu_k^\beta \le \nu_k \le C \mu_k^\beta\} \bigr)\\
&& \hskip-2cm \le  (C-c) \max_{s \in [c, C]}\P\left(\sum_{i=1}^{s \mu_k^\b} X_i \ge t \mu_k^{\a +\frac 12}\right)\P\left(\frac{\nu_k}{\mu_k^\beta}=s\right)
\end{eqnarray*}
We therefore see that
\begin{eqnarray*}
&&\lim_{k \to \infty} \frac 1 {\mu_k^\gamma}\log  \P \bigl( \{Z_{k,\a} \ge t\} \cap\{ c \mu_k^\beta \le \nu_k \le C \mu_k^\beta\} \bigr) \\
&&\hskip2cm=
\lim_{k \to \infty} \frac 1 {\mu_k^\gamma}\log\max_{s \in [c, C]}\P\left(\sum_{i=1}^{s \mu_k^\b} X_i \ge t \mu_k^{\a +\frac 12}\right)\P\left(\frac{\nu_k}{\mu_k^\beta}=s\right).
\end{eqnarray*}
On the other hand, for any $s\in[c,C]$, by the fact that the $(X_i)_i$ are Gaussians, and therefore $\sum_{i=1}^{s \mu_k^\b} X_i$ is
 $\mathcal{N}(0, s\mu_k^\b)$-distributed, standard estimates for Gaussian tail probabilities and the large deviation assumption on $(\nu_k)_k$, Assumption \ref{secondass}, we see that
$$
\limsup_{k \to \infty} \frac 1 {\mu_k^\gamma}\log\P\left(\sum_{i=1}^{s \mu_k^\b} X_i \ge t \mu_k^{\a +\frac 12}\right)\P\left(\frac{\nu_k}{\mu_k^\beta}=s\right)
\le -\frac{t^2}{2s}-I(s),
$$
since the set $\{s\}$ is closed.

For a matching lower bound, we argue locally (as is typical in large deviation theory). For each $s \in [c, C]$ and $\vep>0$ choose an open neighborhood $(s-\delta, s+\delta)$, $\delta>0$, such that $\inf_{t \in (s-\delta, s+\delta)} I(t) \ge I(s)+\vep$ (which is possible due to the lower semi-continuity of $I$). Then
$$
\P\bigl(\{Z_{k,\a} \ge t\} \cap\{ c \mu_k^\beta \le \nu_k \le C \mu_k^\beta\}\bigr)\ge  \P\left(\sum_{i=1}^{s \mu_k^\b} X_i \ge t \mu_k^{\a +\frac 12}\right)\P\left(\frac{\nu_k}{\mu_k^\beta}\in (s-\delta, s+\delta)\right).
$$
Again using the Gaussian tails we see that
\begin{eqnarray*}
\liminf_{k} \frac 1 {\mu_k^\gamma}\log\P\left(\sum_{i=1}^{s \mu_k^\b} X_i \ge t \mu_k^{\a +\frac 12}\right)\P\left(\frac{\nu_k}{\mu_k^\beta}\in (s-\delta, s+\delta)\right)
&\ge& \hspace{-0.3cm}  -\frac{t^2}{2s}-\inf_{t \in (s-\delta, s+\delta)}I(t)\\
& \ge& \hspace{-0.3cm}- \frac{t^2}{2s}-I(s) -\vep.
\end{eqnarray*}
As $\vep >0$  was arbitrary we obtain that
$$
\liminf_{k \to \infty} \frac 1 {\mu_k^\gamma}\log\P\left(\sum_{i=1}^{s \mu_k^\b} X_i \ge t \mu_k^{\a +\frac 12}\right)\P\left(\frac{\nu_k}{\mu_k^\beta}\in (s-\delta, s+\delta)\right)
\ge -\frac{t^2}{2s}-I(s).
$$
Since also $s$ was arbitrarily chosen in $[c,C]$ we arrive at
$$
\lim_{k \to \infty} \frac 1 {\mu_k^\gamma}\log  \P\bigl( Z_{k,\a} \ge t \cap\{ c \mu_k^\beta \le \nu_k \le C \mu_k^\beta\}\bigr)=
-\inf_{s\in[c,C]}  \{t^2/2s+I(s)\}.
$$
Putting things together we see that for each choice of $0<c<C<\infty$ we have
\begin{eqnarray*}
&&\lim_{k \to \infty} \frac 1 {\mu_k^\gamma}\log  \P\bigl( Z_{k,\a}\ge t \bigr)\\
&=&\max\{-\inf_{s\in[c,C]} \{t^2/2s+I(s)\},\lim_{k \to \infty} \frac 1 {\mu_k^\gamma}\log ( c\mu_k^\beta e^{-\frac {t^2} 2 \mu^\gamma}), -I(C)+\vep\}.
\end{eqnarray*}
Letting $c \to 0 $ and at the same time $C \to \infty$ yields the desired result for the upper tails when the summands are standard Gaussians.

The lower tail probabilities $\P\bigl( Z_{k,\a}\ge t \bigr)$ are analyzed in exactly the same way, by intersecting again with the events that $\nu_k$ is smaller, larger or about its expected value. Due to the topological structure of $\R$ the lower and upper tail probabilities suffice to give the asserted MDP, see for example Lemma 4.4 in \cite{HuangLiu:2012}.

If now the $(X_i)_i$ are no longer standard Gaussian random variables, but satisfy the assumptions in the theorem, one again analyses the upper and lower tail probabilities in the same way as in the Gaussian situation.
Indeed, e.g. for the upper tail probabilities the decomposition \eqref{zerlegung} stays the same. The third summand again is dominated by $e^{-\mu_k^\gamma (I(C)+\vep)}$ for any $\vep>0$, while the second summand for any $\vep>0$ and $k$ large enough now is majorised by $c\mu_k^\beta e^{-(\frac {t^2} 2 -\vep) \mu_k^\gamma}$. Here instead of the bound for the tail of Gaussian random variables for the sum of the $X_i$ one now employs the corresponding MDP, Theorem \ref{iid}. Analogously, for the middle summand in \eqref{zerlegung} one replaces the Gaussian estimates by Theorem \ref{iid} and gets the MDP for the upper tail probabilities. Again, the lower tail probabilities are treated in the same way. This proves the theorem.
\end{proof}

\begin{example}
The following example is central, not only for this section but also generally in random summation, e.g. basically the entire book by Kalashnikov \cite{Kalashnikov:1997} is devoted to variants of this example. Note however, that there mainly situations with positive summands are considered.

Here we will assume that the summation index $\nu_k$ is geometrically distributed, while the
$(X_i)_i$ still are i.i.d. centered random variables with variance one, such that \eqref{finmomgen} is satisfied. In order to keep the conventional notation we will rename the index of our summation variable and call it $p$. Hence we will assume that $\nu_p$ is geometrically distributed with parameter $p$, such that
$\mu_p:=\E \nu_p=\frac 1p$ and we will consider the situation where $p$ converges to 0.
It is well known (see e.g. \cite{Dobler:2013}), Theorem 3.4) that then $Z_{p,0}:=\frac 1 {\sqrt p} \sum_{i=1}^{\nu_p}X_i$ converges in distribution to a
$\mathop{Laplace}(0, 1/\sqrt 2)$-distribution. Recall that the $\mathop{Laplace}(a, b)$-distribution is absolutely continuous and its Lebesgue density
$f_{a,b}:\R \to \R$ is given by
$$
f_{a,b}(x):= \frac 1{2b} \exp\left(-\frac{|x-a|}{b}\right).
$$
For the moderate deviations we will consider the random variables
$$
Z_{p,\a}:= \frac 1 {p^{\frac 12 +\a}} \sum_{i=1}^{\nu_p}X_i
$$
for any $0 < \a < \frac 12$.
We will see that we are exactly in the situation described in Theorem \ref{randsum-nongauss}. To this end, all we need to check is that Assumption \ref{secondass} is satisfied by the geometrically distributed random variable $\nu_p$. This is indeed the case. It is easy matter to check, that for any $x>0$
$$
\P\left(\nu_p \ge \left(\frac 1 p\right)^{1+\a}x\right) \sim \exp\left( \left(\frac 1 p\right)^\a x\right),
$$
i.e. Assumption \ref{secondass} is satisfied with $\gamma=\alpha$ and $I(x)=x$. According to Theorem \ref{randsum-nongauss}  we thus obtain that for all $0<\a <\frac 12 $ the family of random variables $(Z_{p,\a})_p$ obey an MDP (or LDP) with speed $p^{-\a}$ and rate function
$$
J(y)=\inf\{y^2/2s+I(s), s \in \R^+\}.
$$
It is interesting to compute this rate function explicitly. Standard analysis shows that the infimum is attained for $s=y/\sqrt 2$ which gives
$J(y)= \sqrt 2 y$. This is quite satisfactory, since $-\sqrt 2 y$ exactly describes the asymptotics of $f_{0, 1/\sqrt 2}(x)$ for large $x$ on a logarithmic scale. In other words our random sums with geometrically distributed summation index confirm the folklore in MDP theory, that (up to a minus sign) the rate function of an MDP behaves like the asymptotic expansion of the logarithmic density in a CLT for large values of $x$. In our case, this is particularly nice, as the Central Limit Theorem is non-standard, i.e. the limit distribution is not the Gaussian distribution and therefore the moderate deviation rate function is non-quadratic. Let us summarise this central example:
\begin{theorem}
Let $(X_i)_i$ be a sequence of  i.i.d. centered random variables with variance one, such that \eqref{finmomgen} is satisfied. Let
$\nu_p$ be geometrically distributed with parameter $p$, such that $\mu_p:=\E \nu_p=\frac 1p$. Then  $(Z_{p,\a})_p$ satisfies a MDP with speed with speed $p^{-\a}$ and rate function $J(y) = \sqrt 2 y$.
\end{theorem}
\end{example}

\section{Large Deviations}
The above results automatically raise the question for the large deviation behavior of random sums. We will start with a situation of the type studied in Section 3. More precisely, we ask for the exponential decay of the probabilities
$\P\bigl( Z_{k,1/2}\ge x \bigr)$ where $Z_{k,1/2}$ is defined by setting $\a=\frac 12$ in \eqref{Zka} and we impose the conditions of Assumption \ref{firstass} on the summation variable.
At the end of the section we will also briefly argue that Section 4 already gives an idea what happens in the situation of Assumption \ref{secondass}.

First of all, Assumption \ref{firstass} needs to be modified. We impose
\begin{ass}\label{fourthass}
Assume that the sequence of random variable $(\nu_k/\mu_k)_k$ obeys an LDP with speed $\mu_k$ with a rate function $I$ that has a unique minimum $I(x)=0$ for $x=1$.
\end{ass}

Also for the $(X_i)_i$  we will need a stronger assumption that relates them to the behaviour of the variables $(\nu_k)_k$
\begin{ass}\label{thirdass}
Suppose that the $(X_i)_i$ are i.i.d. centered random variables that have variance 1 (without loss of generality), and satisfy
\begin{equation} \label{allfinmomgen}
\Lambda_X(t):=\log \E\bigl( \exp(t X_1) \bigr) < \infty
\end{equation}
for all $t>0$.
Moreover assume that the functions $f_t:\R \to \R $ defined by
$$f_t(x):= \Lambda_X(t)x$$
satisfy either
$$
\lim_{M \to \infty} \limsup_{k \to \infty} \frac 1 {\mu_k} \log \E \biggl(e^{\mu_k f_t(\nu_k/\mu_k)}
\mathbbm{1}_{\{ f_t(\frac{\nu_k}{\mu_k})\ge M \}} \biggr)=-\infty
$$
for all $t$, or that for all $t$ there exists some $\gamma >1$ such that
$$
\limsup_{k \to \infty} \frac 1 {\mu_k} \log \E \bigl(e^{\gamma \mu_k f_t(\nu_k/\mu_k)}\bigr)<\infty.
$$
Note that in the above conditions the expectation refers to the random variable $\nu_k$.
\end{ass}

\noindent
Under these conditions we show
\begin{theorem}\label{ldp}
Suppose that Assumption \ref{fourthass} is fulfilled and Assumption \ref{thirdass} is satisfied as well. Then the sequence of random variables $(Z_{k,\frac 12})_k$ satisfies an LDP with speed $\mu_k$ and rate function $J$. $J$ is given by the formula
$$
J(y)= \sup_{\lambda \in \R}[\lambda y - \Gamma(\lambda)]
$$
and $\Gamma$ is given by
\begin{equation}\label{Gamma}
\Gamma(\lambda)=\sup_{x\in \R}[f_\lambda(x)-I(x)].
\end{equation}
\end{theorem}

\begin{proof}
We will again employ the G\"artner-Ellis theorem for this proof. Note however, that the replacement trick of Lindeberg cannot be applied, since the rate functions depends on the distribution of the $X_i$.
Following the computations in Section 3 we see that
\begin{eqnarray*}
\E e^{t \mu_k Z_{k, \frac 12}}&=& \E e^{t\sum_{i=1}^{\nu_k}X_i}\\
&=& \sum_{n=1}^\infty \P(\nu_k=n) (\E e^{t X_1})^n\\
&=& \sum_{x \ge 0} \P\left(\frac{\nu_k}{\mu_k}=x\right) e^{\mu_k \Lambda_X(t)x}\\
&=& \E e^{\mu_k f_t(N_k)},
\end{eqnarray*}
where we have set $N_k=\nu_k/\mu_k$. The right hand side calls for an application of Varadhan's Lemma (Lemma 4.3.1 in \cite{Dembo/Zeitouni:LargeDeviations}). Our Assumption \ref{thirdass} is however just tailor-made for the application of this Lemma (cf. (4.3.2) and (4.3.3.) in \cite{Dembo/Zeitouni:LargeDeviations}). Therefore
$$
\lim_{k \to \infty}\frac 1 {\mu_k} \log \E e^{t \mu_k Z_{k, \frac 12}}=\Gamma(t)
$$
with $\Gamma$ given by \eqref{Gamma}. The theorem thus follows form an application of the G\"artner-Ellis theorem.
\end{proof}

\begin{example}
As a matter of fact, Theorem \ref{ldp} covers many interesting situations. Suppose, e.g. that the $\nu_k$ are Poisson random variables with a parameter $k$ (which then is also equal to $\mu_k$).
Suppose that the $(X_i)_i$ have a finite cumulant generating function $\Lambda_X(t)$, e.g. suppose they are standard Gaussian random variables, in which case  $\Lambda_X(t)=t^2/2$.
Then for all $t$
\begin{eqnarray*}
\E \bigl( e^{\gamma \mu_k f_t(\nu_k/\mu_k)} \bigr)&=& \E \bigl(e^{\gamma \nu_k \frac {t^2} 2}\bigr)= \sum_{n=0}^\infty e^{\gamma n \frac{t^2}2 }\frac {\mu_k^n}{n!} e^{-\mu_k}
= e^{\mu_k(e^{\gamma \frac{t^2}2}-1)},
\end{eqnarray*}
which implies that
$$
\lim_{k \to \infty} \frac 1 {\mu_k} \log \E \bigl(e^{\gamma \mu_k f_t(\nu_k/\mu_k)}\bigr)=e^{\gamma \frac{t^2}2}-1 <\infty
$$
for all $\gamma$ and $t$.
Moreover $(\nu_k/\mu_k)_k$ satisfies an LDP at speed $k=\mu_k$ with rate function $I(x)=1-x+ x\log x$ for nonnegative $x$ and  $-\infty$, otherwise. Therefore $I$ meets the Assumptions \ref{fourthass} and \ref{thirdass}.
Also $\Gamma$ can be quickly computed.
$$
\Gamma(\lambda)= \sup_{x \in \R}[ x\frac{\lambda^2}2-1+x-x\log x]= e^{\lambda^2/2}-1.
$$
Therefore the critical $\lambda$, where the supremum in the definition of $J$ is attained satisfies $y= \lambda e^{\lambda^2/2}$.  It seems pretty difficult to compute the value of $J(y)$, however.

Interestingly, there is another way to obtain the rate function in this case. Consider the proof of Theorem \ref{randsum-nongauss} again. Observe that the only point in the proof, where we use $\alpha< \frac 12$ there, is when we bound the tail probabilities
$$\P\biggl(\sum_{i=1}^{\mu_k^\b} X_i \ge t \mu_k^{\a +\frac 12}\biggr)$$ by $e^{-(\frac {t^2} 2 -\vep) \mu_k^\gamma}$, i.e. we employ the moderate deviations bound for a fixed number of summands. However, if the $X_i$ are themselves Gaussian random variables such a bound is true {\textit{for all}} $\alpha>0$. We can therefore employ the techniques used in the proof of Theorem \ref{randsum-nongauss} to the situation where $\alpha=\frac 12$, which implies $\beta=\gamma=1$. This is exactly the situation described in this example. We then obtain that with $(\nu_k)_k$ again chosen as a Poisson random variable with parameter $k$ and $(X_i)_i$ independent standard Gaussian random variables, the sequence of random sums $(\frac 1 k \sum_{i=1}^{\nu_k} X_i)_k$ obeys an LDP with speed $k$ and rate function.
$$
J(y)=\inf\{y^2/2s+1-s+s\log s, s \in \R^+\}.
$$
Unfortunately, it is a common issue in large deviations theory to evaluate such rate functions explicitly.

The latter approach can be generalized. To this end one would assume that \eqref{finmomgen} holds for all $\lambda$ which entails that for deterministic summation indices $k$ the sequence $\bigl( \frac{1}{k} \sum_{i=1}^k X_i\bigr)_k$ obeys an LDP with speed $k$ and rate function
$$
K(x)= \sup_{\lambda \in \R} \bigl[ \lambda x - \log \E\bigl(\exp(\lambda X_1)\bigr) \bigr].
$$
Now assume that also the sequence of summation variables $(\nu_k/\mu_k)_k$ obeys an LDP with speed $\mu_k$ and some rate function $I$.  Following the approach described in the proof of Theorem \ref{randsum-nongauss} we then obtain an LDP for the sequence of random sums $\bigl( \frac 1 {\mu_k} \sum_{i=1}^{\nu_k} X_i \bigr)_k$ with speed $\mu_k$ (as expected) and rate function
$$
J(y)=\inf\left\{s K\left(\frac y s\right)+I(s), s \in \R^+\right\}.
$$
\end{example}


\newcommand{\SortNoop}[1]{}\def\cprime{$'$} \def\cprime{$'$}
  \def\polhk#1{\setbox0=\hbox{#1}{\ooalign{\hidewidth
  \lower1.5ex\hbox{`}\hidewidth\crcr\unhbox0}}} \def\cprime{$'$}
\providecommand{\bysame}{\leavevmode\hbox to3em{\hrulefill}\thinspace}
\providecommand{\MR}{\relax\ifhmode\unskip\space\fi MR }
\providecommand{\MRhref}[2]{%
  \href{http://www.ams.org/mathscinet-getitem?mr=#1}{#2}
}
\providecommand{\href}[2]{#2}

\end{document}